\documentclass[11pt]{article}     
\usepackage{amsthm}
\usepackage{bbm}
\usepackage{amssymb}
\usepackage{aistats2020}
\usepackage{algorithm}
\usepackage[noend]{algpseudocode}
\makeatletter
\def\BState{\State\hskip-\ALG@thistlm}

\usepackage{amsmath}               
  {
      \theoremstyle{plain}
      \newtheorem{assumption}{Assumption}
  }
\RequirePackage{filecontents}
\usepackage{algorithm}
\usepackage[noend]{algpseudocode}
\makeatletter
\def\BState{\State\hskip-\ALG@thistlm}
\makeatother

\usepackage[noend]{algpseudocode}
\newcommand{\e}{\varepsilon}
\newcommand{\la}{\langle}
\newcommand{\ra}{\rangle}

\newcommand{\E}{\mathbb E}
\newcommand{\R}{\mathbb R}
\newcommand{\Var}{\text{Var}}

\newcommand{\de}{\delta}

\newcommand{\Prob}{\mathbb P}
\usepackage[utf8]{inputenc} 
\usepackage[T1]{fontenc}    
\usepackage{hyperref}       
\usepackage{url}            
\usepackage{booktabs}       
\usepackage{amsfonts}       
\usepackage{nicefrac}       
\usepackage{microtype}      
\usepackage{cleveref}
\usepackage{graphbox}
\usepackage{graphicx}
\usepackage[dvipsnames]{xcolor}
\usepackage{import}
\usepackage{multicol, blindtext}
\usepackage[font=small,labelfont=bf]{caption}
\usepackage{natbib}
\usepackage{amssymb,amsmath,amsthm,latexsym}

\usepackage{booktabs}

\usepackage{fullpage}
\usepackage{hyperref}
\usepackage{url}
\usepackage{multicol}
\usepackage{bbm}

\usepackage{graphicx}

\newtheorem{theorem}{Theorem}
\newtheorem{lemma}{Lemma}

\newtheorem{corollary}{Corollary}

\newcommand{\be}{\begin{equation}}
\newcommand{\ee}{\end{equation}}

\begin{document}
	\twocolumn[
\aistatstitle{Log Barriers for Safe Non-convex Black-box Optimization}


 
\aistatsauthor{Ilnura Usmanova \And  Andreas Krause \And  Maryam Kamgarpour }
\aistatsaddress{ Automatic Control Laboratory,\\
 ETH Zürich, Switzerland \And  Machine Learning Institute,\\
 ETH Zürich, Switzerland \And Automatic Control Laboratory,\\
 ETH Zürich, Switzerland
 } 
]
	\begin{abstract}
		We address the problem of minimizing a  smooth function $f^0(x)$ over a compact set $D$ defined by smooth functional constraints $f^i(x)\leq 0,~ i = 1,\ldots, m$ given noisy value measurements of $f^i(x)$.
		This problem arises in safety-critical applications, where certain parameters need to be adapted online in a data-driven fashion, such as in
		personalized medicine, robotics, manufacturing, etc. In such cases, it is important to ensure constraints are not violated while 
		taking measurements and seeking the minimum of the cost function. We propose a new algorithm s0-LBM, which provides provably feasible iterates with high probability and applies to the challenging case of uncertain zero-th order oracle. We also analyze the convergence rate of the algorithm, and empirically demonstrate its effectiveness.
\footnote{We thank the support of Swiss 
National Science Foundation, under the grant SNSF 200021\_172781, and ERC under the European Union's Horizon 2020 research and innovation programme 
grant agreement No 815943.}
	\end{abstract}

	\section{INTRODUCTION}
	Many applications in robotics \citep{schaal2010learning}, manufacturing \citep{maier2018turning}, health sciences, finance, etc. require minimizing a loss function under constraints and uncertainty.  Optimizing a loss function under partially revealed constraints can be further complicated by the fact that observations are available only inside the feasible set. Hence, one needs to carefully choose actions to ensure the feasibility of each iterate while pursuing the optimal solution. In the machine learning community, this problem is known as \emph{safe learning}. For such tasks, feasible optimization methods are required.  There are many first and second order feasible methods in the literature. Although given noisy zero-th order oracle the Hessians are hard to estimate with good accuracy, we can approximate derivatives using finite differences.	Most well known and widely used first order methods for stochastic optimization are dealing with constraints using projections. However, the lack of 
	global knowledge of the constraint functionals makes it impossible to compute the corresponding projection operator. 

\paragraph{Related work.}
	There is a lack of zero-th order feasible (safe) algorithms for black-box constrained optimization in the literature.
	 \citet{balasubramanian2018zeroth} provide a comprehensive analysis of the performance of several zero-th order algorithms for non-convex optimization. However, 
	the conditional gradient based algorithm (Frank-Wolfe) of \citet{balasubramanian2018zeroth} for constrained non-convex problems requires global knowledge of the constraint functions, as the linear objective must be optimized with respect to these constraints. 
	 \citet{usmanova2019safe} introduce a Frank-Wolfe based algorithm applied to the case of noisy zero-th order oracle for linear constraints. 
	 This work proves the feasibility of iterates with high probability and bounded the convergence rate with high probability, but does require convexity.

 Non-convex non-smooth problems can be addressed by feasible methods, such as the first order Method of Feasible Directions \citep{topkis1967convergence} or the second order Feasible Sequential Quadratic Programming (FSQP) algorithm
	\citep{tang2014feasible}. Another algorithm for non-convex non-smooth problems is given by
	\citet{facchinei2017feasible}. The idea of this algorithm is to use as a direction of movement the minimizer of a local convex approximation of the objective subject to a local convex approximation of constraint set. 
	Unfortunately, all the guarantees for the above methods are in terms of asymptotic convergence to a stationary point.

Another class of safe algorithms for global black-box optimization is based on Bayesian Optimization(BO) such as SafeOpt \citep{sui2015safe} and its extensions \citep{berkenkamp2016bayesian}. The main drawback of these methods is that the computational complexity 
grows exponentially with dimensionality.

	Interior Point Methods (IPM) are feasible approaches by definition, and they are widely used for Linear Programming, Quadratic Programming, and Conic optimization problems. 
	By using self-concordance properties of specifically chosen barriers and second order information, these problems are  shown to be extremely efficiently solved by IPM. However, in the cases when constraints are unknown, building the barrier with self-concordance properties is not possible. In these cases it is possible to use logarithmic barriers for general black-box constraints. \citet{hinder2018one} propose to  choose adaptive step sizes for the  gradient algorithm for the log barriers and give the analysis of the convergence rate. The work of \citet{hinder2018one} assume knowledge of the exact gradients of the cost and constraint functions. In the present work, we extend this approach for the case in which we only have a possibly noisy zero-th order oracle.
	
\paragraph{Our Contributions.} In this paper we propose the first safe zero-th order algorithm 
for non-convex 
 optimization with a  black-box noisy oracle. We prove that it generates feasible iterations  with high probability, and analyse its convergence rate to the local optimum. Each iteration is  computationally very cheap and 
does not require solving any subproblems such as those required for Frank-Wolfe or Bayesian Optimization based algorithms.
In Table 1, we provide a comparison of our algorithm with the existing ones for unconstrained and constrained zero-th order non-convex optimization. 
In the first two algorithms a multiple point feedback is assumed, i.e., it is possible to measure at several points with the same noise realization. 
 The convergence rate in the second column is proven for known polyhedral constraints.  
 In our algorithm we assume a more realistic and more complicated setup where the noise is changing with each measurement. 
There are some works in zero-th order 1-point feedback for convex optimization with known constraints, for example \citet{bach2016highly} require $O\left(\frac{d^2}{\eta^3}\right)$ measurements to achieve $\eta$ accuracy.
\begin{table*}[t]\label{T1}
 		\begin{center}
 		\scalebox{0.9}{
 		\begin{tabular}[h!]{|p{20mm}|p{35mm}|p{50mm}|p{65mm}
 			|}
 				\hline
 				 Problem
 				& Unconstrained
				& Known  constraints
				& {\bf Safe despite unknown constraints} \\
				\hline
 				 Feedback
 				& Noisy 2-point
				& Noisy 2-point
				& Noisy 1-point \\
				\hline
				Optimality criterion &  Stationary point: \newline $\E\|\nabla f(x_t)\|_2 \leq \eta
 				$ &
 				$\eta$-stationary point: \newline $\E\la \nabla f(x_t), x_t - u\ra \leq \eta \forall u\in D$ 
 				& $\eta$-approximate scaled KKT point: \newline
				 $\Prob\{\|\nabla L(x_t)\|_2 \leq  \eta(1+\|\lambda_t\|_{\infty})\}\geq 1-\de$\\
				\hline
				Number of measurements
				&  $O\left(\frac{d}{\eta^4
				}\right)$ \newline\footnotesize\textit{(no matrix inversion)}\small
				\newline 
				 $O\left(\frac{d}{\eta^{3.5}}
				\right) + \tilde O \left(\frac{d^4}{\eta^{2.5}}
				\right) $ \newline\footnotesize\textit{(hessian estimation)}\newline 
				(Balasubramanian, Ghadimi, 2018) \small
				& $ O\left(\frac{d}{\eta^4
				}\right)$ \newline
		\footnotesize\textit{(conditional gradient based,
				\newline requires $O\left(\frac{1}{\eta^2
				}\right)$ optimization subproblems w.r.t. $D$)}\newline
				(Balasubramanian, Ghadimi, 2018) \small
				& $\tilde O\left(\frac{d^{3}}{\eta^{7}}\sqrt{\ln 1/\de}
				\right) 
				$ ($ O\left(\frac{1}{\eta^3}
				\right) steps
				$) \newline\footnotesize\textit{({\bf this paper},  \newline does not require solving subproblems or matrix inversions)} \small\\
				
				\hline
                \end{tabular}
	}
	\end{center}
			\caption{
			Upper $O(\cdot)$ bounds on number of  zero-th order oracle calls for non-convex smooth optimization.}
\end{table*}


\section{PROBLEM STATEMENT}
\paragraph{Notations and definitions.} 
Let $\|\cdot\|_2$,$\|\cdot\|_1$ and $\|\cdot\|_{\infty}$ denote $l_2$-norm, $l_1$-norm  and $l_{\infty}$-norm respectively on $\R^d$. A function $f(x)$ is called \textit{$L$-Lipschitz continuous} if
\vspace{-0.2cm}
\begin{align}
|f(x) - f(y)|\leq L\|x - y\|_2 
.
\end{align} 
It is called \textit{$M$-smooth} if 
the gradients $\nabla f(x)$ are $M$-Lipschitz continuous, i.e., 
\vspace{-0.2cm}
\begin{align}\label{M}
\|\nabla f(x) - \nabla f(y)\|_2\leq M\|x - y\|_2. 
\end{align}
A random variable $\xi$ is zero-mean $\sigma$-sub-Gaussian if  
\vspace{-0.3cm}
\begin{align*}\forall \lambda \in \R ~ 
\E\left[ e^{\lambda \xi}\right] \leq \text{exp}\left(\frac{\lambda^2\sigma^2}{2}\right),
\end{align*} which implies that 
$\Var\left[\xi 
\right] \leq \sigma^2$.
 \paragraph{Problem formulation} We consider the problem of non-convex \textit{safe learning}  defined as a constrained optimization problem 
 \hspace{-0.2cm}
\begin{align}\label{problem}
  &\min_{x\in \R^d} ~~~~~~~f^0(x) \nonumber\\
 &\text{subject to }  f^i(x) \leq 0, ~~i =  1,\ldots,m,
\end{align}
where the objective function $f^0: \R^d \rightarrow \R$ and the constraints $f^i: \R^d \rightarrow \R$ are unknown 
continuous functions, and can only be accessed at feasible points $x$. We denote by $D$ the feasible set $D := \{x\in\R^d: f^i(x)\leq  0, i = 1,\ldots,m\}.$ 

\begin{assumption}\label{A:1}
Let $M,L > 0$. 
The objective and the constraint functions $f^i(x)$ for $ i = 0,\ldots,m$ are  $M$-smooth on $D$. Also, 
 constraint  functions $f^i(x) $ for $  i = 1,\ldots,m$ are $L$-Lipschitz continuous on $D$.
\end{assumption} 
\begin{assumption}\label{A:2}
The feasible set $D$ has a non-empty interior, and there exists a known starting point $x_0$ for which $f^i(x_0)<0$  for $i = 1, \ldots, m.$
\end{assumption} 


 \paragraph{Oracle information.}
We consider zero-th order oracle information.  If we can measure the function values exactly, then one possible oracle is the \emph{Exact Zero-th Order oracle (EZO)}.\\
\textbf{EZO:} 
provides the exact values  $f^i(x_t)~ \forall i = 0,\ldots, m$  for any requested  
point $x_t \in D.$ 

In many applications, the measurements of the functions are noisy.  
We assume that the additive noise $\xi$ is coming from a zero-mean $\sigma$-sub-Gaussian distribution. 
We call this oracle \emph{Stochastic Zero-th Order oracle (SZO)}.\\
\textbf{SZO:} $\forall i = 0,\ldots, m$ provides the noisy function 
values $F^i(x_t, \xi^i_t) = f^i(x_t) + \xi^i_t$, $\xi^i_t$ is a zero-mean $\sigma$-sub-Gaussian noise independent of previous measurements 
for any requested point $x_t\in D$. 
We assume that the noise values $\xi^i_t$ are independent over time $t$ and indices $i$.

\paragraph{Optimality criterion.}
The condition $\|\nabla f^0(x_T)\|_2\leq \e$ is usually used as an optimality criterion in non-convex smooth optimization without constraints. It is well known that in the unconstrained case the classical gradient descent method converges with rate $O(\frac{1}{\e^2}),$ which matches the lower bound derived for this class of problems \citep{carmon2017lower}.  In the non-convex constrained optimization, first order criteria are Karush-Kuhn-Tucker (KKT) conditions, which are necessary in the presence of 
regularity conditions called Constraint Qualification (CQ).
In such cases, we can measure the solution accuracy by satisfying  approximate KKT conditions. 
The point $(x, \lambda) \in \R^n\times \R^m$ is called a KKT  point if it satisfies the necessary condition for local optimality:
\vspace{-0.2cm}
\begin{align*}
&\lambda^i,-f^i(x) \geq 0, ~ \forall i = 1,\ldots,m\\
&\lambda^i(-f^i(x)) = 0,~ \forall i = 1,\ldots,m\\
&\|\nabla L(x, \lambda)\|_2 = 0,
\end{align*}
where $\lambda$ is the vector of dual variables and  $L(x,\lambda) := f^0(x) + \sum_{i=1}^m\lambda^i f^i(x)$ is the Lagrangian. 
There are several ways to define an approximate KKT point, 
 see \citep{cartis2014complexity,birgin2016evaluation}. 
 Similar to \citet{cartis2014complexity}, we define an approximate KKT point as follows. 
An \emph{$\eta$-approximate scaled KKT point} \textbf{(s-KKT)} for some $\eta>0$ is a pair $(x,\lambda)$ which satisfies: 
\begin{align}
\vspace{-0.3cm}
&\lambda^i,-f^i(x) \geq 0,~ \forall i = 1,\ldots,m \tag{s-KKT.1}\\
&\lambda^i(-f^i(x)) \leq \eta ,~ \forall i = 1,\ldots,m \tag{s-KKT.2}\\
&\|\nabla L(x, \lambda)\|_2 \leq \eta (1+\|\lambda\|_{\infty}). \tag{s-KKT.3}
\end{align}
 Note that the approximation lies in substituting the equality constraints of the standard KKT with inequalities (s-KKT.2), (s-KKT.3). In case $\|\lambda\|_{\infty}$ can be uniformly bounded by some constant $Q$, the scaled $\eta$-approximate KKT point is an unscaled $(Q+1)\eta$-approximate KKT point.

\section{PRELIMINARIES}
\subsection{ Log barrier algorithm.}
We 
 address the safe learning problem using the log barriers approach. 
 The log barrier function and its gradient are defined as follows:
 \vspace{-0.3cm}
\begin{align}\label{log-b}
B_{\eta}(x)& = f^0(x) - \eta \sum_{i=1}^m \log (-f^i(x)),\\
  \vspace{-0.3cm}\label{d-log-b}
\nabla B_{\eta}(x)& = \nabla f^0(x) - \eta \sum_{i=1}^m \frac{\nabla f^i(x)}{f^i(x)}.
\end{align} 
The main idea of the log barrier methods  is to solve a sequence of the \textit{barrier subproblems}
 \vspace{-0.2cm}
\begin{align}\label{barrier:sub}\min_{x\in \R^d} B_{\eta_k}(x)\end{align}
with decreasing $\eta_k = \eta_{k-1}/\mu$ for some $\mu > 1.$  For the rest of the paper we fix $\eta$ to be constant, since we propose an algorithm to solve the subproblems, which is challenging given only noisy function evaluations. 
Let us set the pair of primal and dual variables to be $(x, \lambda) = \left(x, \left[\frac{\eta}{- f^i(x)}\right]\right)$, where $\lambda^i = \frac{\eta}{- f^i(x)}$ for $i = 1,\ldots,m.$ 
Then, one can verify the following properties: 
\vspace{-0.3cm}
\begin{enumerate} 
	\item $\|\nabla L(x,\lambda)\|_2 = \|\nabla B_{\eta}(x)\|_2;$ 
	\vspace{-0.3cm}
\item $\lambda_t(-f^i(x)) = \frac{\eta}{- f^i(x)}(-f^i(x)) = \eta;$
\vspace{-0.3cm}
\item $\lambda, -f^i(x) \geq 0.$
\end{enumerate}
\vspace{-0.4cm}
Now recall (s-KKT.1)-(s-KKT.3) and suppose that $x$ is such that $\|\nabla B_{\eta}(x)\|_2\leq \eta.$ Then, one can verify that the pair $(x, \lambda)$ 
is an $\eta$-approximate scaled KKT point.

\paragraph{First order log barrier algorithm \citep{hinder2019poly}. }
For structured problems like conic  optimization with known self-concordant barriers and computable Hessians, to solve the log barrier subproblems (\ref{barrier:sub}) barrier methods classically use the Newton algorithm. However, since we assume that the structure is unknown, and  the second order information is inaccessible, we would like to use gradient descent type methods with the step direction $g_t' = -\nabla B_{\eta}(x_t)$  to find a solution of the  barrier subproblem (\ref{barrier:sub}). The main drawback of solving and analysing the log barrier subproblems using gradient methods is that the log barriers themselves are non-Lipschitz continuous and non-smooth functions since their gradients grow to infinity on the boundary. This might lead to unstable behaviour of the gradient based algorithm close to the boundary, and the step sizes have to be chosen exponentially small.  To handle this drawback, 
in the paper \citep{hinder2019poly} the authors proposed to choose an adaptive step size $\gamma_t  =
 \min \left\{ \frac{\min\{-f^i(x_t)\}}{2L\|g_t'\|},\frac{1}{L_2(x_t)}\right\}$, 
where $L_2(x_t)$
represents a local Lipschitz constant of  $\nabla B_{\eta}(x)$ at the  point $x_t$. The convergence rate for finding the solution of the subproblem $B_{\eta}(x)$, i.e., convergence to an $\eta$-approximate KKT point using their algorithm with adaptive step sizes is formulated in Theorem 1.
\begin{theorem}(Claim 2. \citep{hinder2019poly}) Under Assumption \ref{A:1} for any constant $\tau_l>0$, after at most $T$ iterations such that  
	\vspace{-0.3cm}
	$$T\leq  2\frac{B_{\eta}(x_0) - \min_{x\in D}B_{\eta}(x)}{\min\left\{\frac{\tau_l\eta^2}{2L},\frac{\tau_l^2\eta^2}{4(M + \frac{L^2}{\eta})}\right\}}$$
	the procedure finds a point $(x_k,\lambda_k)$ such that 
		\vspace{-0.3cm}
	$$\|\nabla  B_{\eta}(x_k)\|_2\leq 
	 \tau_l\eta\left(1+\|\lambda_k\|_1\right)
	.$$
\end{theorem}
\citet{hinder2019poly} were first  to derive such rates for first order log barrier methods for general $L$-Lipschitz, $M$-smooth functions because of their adaptive step size. This choice 
allowed them to define and use local Lipschitz constant of the barrier gradients for their analysis.

\paragraph{Idea of our algorithm.} We extend the first order algorithm of \citet{hinder2019poly} to the zero-th order oracle case. To this end, we propose to estimate the gradients from zero-th order information using finite differences. Based on these estimates we can estimate the barrier gradients. At the same time, we ensure that the measurements are taken in the safe region despite lack of knowledge of the constraint functions.

\subsection{Zero-th order gradient estimation.} \label{sec:oracle}
We now construct  estimators $G^i(x_t, \nu_t)$ and  $\hat G^i(x_t, \nu_t, \xi_t)$ of the gradient $\nabla f^i(x_t)$ for each $i = 1,\ldots,m$, based on the zero-th order information provided by EZO and SZO, respectively. We denote the differences between the estimators and the function gradient by $\Delta^i_t$ and $\hat \Delta^i_t$:
\vspace{-0.2cm}
\begin{align}\label{delta1}
&\Delta^i_t = G^i(x_t, \nu_t) - \nabla f^i(x_t)\\
\label{delta2}
&\hat \Delta^i_t = \hat G^i(x_t, \nu_t, \xi) -\nabla  f^i(x_t).
\vspace{-0.3cm}
\end{align} 
\paragraph{EZO estimator of the gradient.} For the exact oracle, we take $d$  measurements around the current point to estimate $\nabla f^i(x_t)$. 
Let
$e_j$ be the $j$-th coordinate vector. For $\nu_t > 0$ 
we can use the following 
estimator of $\nabla f^i(x_t)$ \footnote{Another way is to measure at $x_t$ and at $x_t + \nu_t u_t$, where  $u_t$ 
is a random direction and use them for stochastic zero-th order estimator. In that case the dependence on dimensionality will be smaller, however the dependence on the confidence level $\de$ will be much worse:  $O(\frac{d}{\de})$ instead of $O(d^2\sqrt{\ln \frac{1}{\de}})$ for our algorithm, since it then can be obtained only with multi-starts using Markov inequality (c.r. \citet{ghadimi2013stochastic}, \citet{yu2019zeroth}). }:
\vspace{-0.3cm}
\begin{align}\label{est:1}
G^i(x_t, \nu_t) =  \sum_{j = 1}^d\frac{f^i(x_t + \nu_t e_j) - f^i(x_t)}{\nu_t}e_j.
\vspace{-0.3cm}
\end{align}
Using (\ref{M}), we can upper-bound the deviation in (\ref{delta2}) of this estimator from $\nabla f^i(x_t)$ as follows (see Appendix I for the proof):
\vspace{-0.3cm}
\begin{align}\label{bias1}
&\|\Delta_t^i\|_2
\leq  
\frac{\sqrt{d}\nu_t M}{2}.
\end{align}
\paragraph{SZO estimator of the gradient.}
For the stochastic oracle, we take $(d+1)n_t$ measurements around the current point $x_t$. The number of measurements $n_t$ needs to be chosen dependent on $\nu_t$ since the influence of the noise variance increases with the decreasing distance of measurements from each other. 
We define the vector of noises $\xi^i_{l} = \{\xi^i_{0l},\ldots,\xi^i_{dl}\} \in \R^{d+1},$ where $F^i(x_t + \nu_t e_j,\xi^i_{jl}) =  f^i(x_t+ \nu_t e_j) + \xi^i_{jl}, F^i(x_t,\xi^i_{0l}) =  f^i(x_t) + \xi^i_{0l}.$ Then, the estimator $\hat G^i(x_t, \nu_t, \xi^i_{(n_t)})$ with $\xi^i_{(n_t)} = \{\xi^i_l\}_{l = 1,\ldots,n_t}$ is given by
\vspace{-0.3cm}
\begin{align}\label{est:2}
\hspace{-0.3cm}\hat G^i(x_t, \nu_t, \xi^i_{(n_t)}) = \frac{\sum_{l = 1}^{n_t}\hat G^i_0(x_t, \nu_t, \xi^i_l)}{n_t}, 
\end{align} 
\vspace{-1cm}
\begin{align*}
&\hat G^i_0(x_t, \nu_t, \xi^i) = \sum_{j = 1}^d\frac{F^i(x_t + \nu_t e_j, \xi^i_j)- F^i(x_t,\xi^i_0)}{\nu_t}  e_j.
\end{align*}
\begin{lemma}\label{lemma:1}The deviation in (\ref{delta2}) is bounded as:
\vspace{-0.1cm}
	\begin{align*}
	&\Prob \left\{\|\hat \Delta^i_t\|_2   \leq  \sqrt{\frac{d\nu_t^2 M^2}{4}  + 2d\sigma^2 \frac{\ln 1/\delta}{n_t\nu_t^2}}\right\} \geq 1-\de.
	\end{align*}
	To balance the above two terms inside the square root, we can choose $n_t=\frac{8\sigma^2 \ln \frac{1}{\delta}}{3\nu_t^4 M^2}$, then $\Prob \left\{\|\hat \Delta^i_t\|_2  \leq  \frac{\sqrt{d}\nu_t M}{2}\right\} \geq 1-\de.$
\end{lemma}
For the proof see Appendix F.
\section{ALGORITHM}
Having computed the gradient estimators for both exact and noisy oracles, we devise our proposed safe oracle-based optimization algorithm. 
\subsection{ Safe zero-th order log barrier algorithm for EZO. } 
Using (\ref{d-log-b}) and (\ref{est:1}), we can estimate $\nabla B_{\eta}(x_t)$ by: 
 \vspace{-0.3cm}
\begin{align}\label{g_t_EZO}
g_t=  G^0(x_t, \nu_t)  + \eta \sum_{i=1}^t\frac{G^i(x_t, \nu_t) }{-f^i(x_t)}.
\end{align}
Motivated by 
\citet{hinder2019poly}, we define the local smoothness constant of $B_{\eta}(x_t)$
  \vspace{-0.2cm}
\begin{align}\label{L_2} 
\hspace{-0.1cm}L_2(x_t) = M+\sum _{i=1}^m\frac{2\eta M}{- f^i(x_t)} + \frac{4 \eta L^2 }{ (- f^i(x_t))^2}.
\end{align}
It is based on bounding the change of $\nabla B_{\eta}(x_t)$ on the region restricted by the step size $\gamma_t$ and is used for obtaining the convergence rate. Then, we propose the following algorithm:
 \vspace{-0.3cm}
\begin{algorithm}[H]
	\caption{Safe Logarithmic Barrier method with EZO \textbf{(0-LBM)}}\label{safe-barrier-0}
	\small
	\begin{algorithmic}[1]
		\BState \emph{Input:} $x_0 \in D$, number of iterations $T$, $\eta > 0$ 
		, $L,M > 0$; 
		\While {$t \leq T$}
		\State Take $\nu_t = \min\left\{\frac{\eta}{\sqrt{d}M}, \frac{\min\{-f^i(x_t)\}}{\max\{L,m\sqrt{d}M\}}\right\};$
		\State Make $d+1$ calls to EZO for $f^i(x), i = 0,\ldots,m$ at the points $x_t, x_t + \nu_t e_j, j = 1,\ldots,d;$
		\State Compute an estimator $g_t$ of  $\nabla B_{\eta}(x_t)$ (\ref{g_t_EZO}), $L_2(x_t)$ (\ref{L_2}) ;
		\State $\gamma_t = \min\left\{
		\frac{\min_i\{- f_i(x_t)\}}{2L\|g_t\| },\frac{1}{L_2(x_t)}\right\}
		$;
		\State $x_{t+1} = x_{t} - \gamma_t g_t$, $\lambda_{t+1} =\left[\frac{\eta}{-f^i(x_{t+1})}\right]$;
		\EndWhile
	\end{algorithmic}
\end{algorithm}

\subsection{ Safe zero-th order log barrier algorithm for SZO. }
In case of a noisy zero-th order oracle, we construct the upper confidence bound $\hat f^i_{\de}(x_t)$  for $f_i(x_t)$ based on $(d+1)n_t$ measurements taken around $x_t$ during the algorithm iterations. 
\begin{lemma}\label{lemma:ucb}
$\Prob\left\{ \hat f^i_{\de}(x_t) - \e_t \leq  f^i(x_t)  \leq  \hat f^i_{\de}(x_t)\right\}\geq 1-\de,$
 	 where
  $\e_t = 2 \frac{\sigma \sqrt{\ln \frac{1}{\de}}}{\sqrt{n_t}}$ and 
 \vspace{-0.3cm}
\begin{align}\label{f_up2_SZO}
\hspace{-0.3cm}
\hat f^i_{\de}(x_t, \xi_{(n_t)}^i)=\frac{\sum_{l = 1}^{n_t} F^i(x_t, \xi^i_{0l})}{n_t}+\frac{\sigma \sqrt{\ln \frac{1}{\de}}}{\sqrt{n_t}}
\end{align}
	\end{lemma}
	For the proof see Appendix D.\\ 
 We construct an estimator of 
$\nabla B_{\eta}(x_t)$ as $\hat g_t$:
\vspace{-0.3cm}
\begin{align}\label{g_t_SZO}
\hspace{-0.4cm}
\hat g_t = \hat G^0(x_t, \nu_t, \xi^0_{(n_t)})  + \eta \sum_{i=1}^t\frac{\hat G^i(x_t, \nu_t, \xi^i_{(n_t)}) }{-\hat f^i_{\de}(x_t,\xi^i_{(n_t)})}.
\end{align}
An upper confidence bound on  
$L_2(x_t)$  is then 
\vspace{-0.2cm}
\begin{align}\label{L_2:ubc}
\hspace{-0.4cm}\hat L_2(x_t) = M+\sum _{i=1}^m\frac{2\eta M}{-\hat f^i_{\de}(x_t)} + \frac{4 \eta L^2}{ (-\hat f^i_{\de}(x_t))^2}.
\end{align}
The algorithm for SZO is as follows:
\vspace{-0.1cm}
\begin{algorithm}[H]
	\caption{Safe Logarithmic Barrier method with SZO \textbf{(s0-LBM)}}\label{safe-barrier-2}
	\small
	\begin{algorithmic}[1]
		\BState \emph{Input:} $x_0 \in D$, $T$, $\eta > 0$, $L,M > 0$, $\sigma > 0$, $\de>0$;
		\While {$t \leq T$}
		\State  $\nu_t = \min\left\{\frac{\eta}{\sqrt{d}M}, \frac{\min\{-f^i(x_t)\}}{\max\{L,m\sqrt{d}M\}}\right\}$, $n_t = \frac{8\sigma^2 \ln \frac{1}{\delta}}{3\nu_t^4 M^2}$;
		\State Make $n_t$ calls to SZO for $f^i(x), i = 0,\ldots,m$ at each point $x_t, x_t + \nu_t e_j, j = 1,\ldots,d;$
		\State Compute the upper confidence bounds $\hat f_i(x_t)$ (\ref{f_up2_SZO}) and the estimator $\hat g_t$ of $\nabla B_{\eta}(x_t)$ (\ref{g_t_SZO}), $L_2(x_t)$ (\ref{L_2});
		\State $\gamma_t = \min \left\{\frac{\min_i\{- \hat f^i_{\de}(x_t)\}}{2L\|\hat g_t\|_2}, \frac{1}{\hat L_2(x_t)}\right\}$; 
		\State $x_{t+1} = x_{t} - \gamma_t \hat g_t $, $\hat \lambda_{t+1} =\left[\frac{\eta}{-\hat f^i_{\de}(x_{t+1})}\right]$;
		\EndWhile
	\end{algorithmic}
\end{algorithm}
\section{SAFETY AND CONVERGENCE}
\subsection{Safety of 0-LBM with EZO}
To guarantee the safety of $x_t$, the 
step size $\gamma_t $ for barrier gradient step $g_t$ is restricted using Lipschitzness of the constraints.

\begin{lemma}\label{lemm:safety1}
For $\nu_{t}\leq \frac{\min_i-f^i(x_{t})}{L}$ the measurements taken around 
$x_t$ at Step 4 of 0-LBM are feasible. Moreover, if the step size $\gamma_t$ for step $g_t$ is such that $\gamma_t \leq \frac{\min\{-f^i(x_t)\}}{2L\|g_t\|}$  then $f^i(x_{t+1}) \leq \frac{1}{2}f^i(x_{t})\leq 0$, i.e., $x_{t+1}$ is feasible.
\end{lemma}
\begin{proof}
For any point $y\in \R^d$ satisfying $\|x_t-y\| \leq \frac{\min_i-f^i(x_t)}{2L}$ we have
$\forall i = 1,\ldots,m,~ f^i(y) \leq $
\begin{align*}
 &\leq   f^i(x_{t}) + \la \nabla f^i(x_{t}) , y - x_t\ra \leq  f^i(x_{t}) + L \|y - x_t\| \\
& \leq  f^i(x_{t}) + L \frac{-f^i(x_t)}{2L}
\leq \frac{1}{2}f^i(x_{t}) \leq 0.
\end{align*}
The statement of the lemma follows directly.
\end{proof}
\subsection{Safety of s0-LBM with SZO}
From Step 5 of the algorithm recall that $\gamma_t \leq \frac{1}{t^{1/3}}\frac{\min_i\{- \hat f^i_{\de}(x_t)\}}{2L \|\hat g_t\|}$ and from \Cref{lemma:ucb} that $\Prob\{\hat f^i_{\de}(x_t) \geq f^i(x_t) \} \geq 1-\de.$ Thus, directly from \Cref{lemm:safety1} 
we obtain the following result.
\begin{lemma}
     For any $x_t$ generated by s0-LBM with probability $\geq 1-\de$ it is true that $f^i(x_{t+1}) \leq 
     \frac{1}{2}f^i(x_{t}) \leq 0$. Moreover, given $\nu_{t+1} \leq \frac{\min_i\{-\hat f^i_{\de}(x_t)\}}{2L}$ the measurements around the next point $x_{t+1}$ 
      at Step 4 of s0-LBM are also feasible with probability $\geq 1-\de$.
\end{lemma} 




\subsection{Convergence of 0-LBM and s0-LBM 
}


Let us denote  $\zeta_t  := \nabla  B_{\eta}(x_t)-g_t$ and $\hat \zeta_t := \nabla  B_{\eta}(x_t)-\hat g_t$, as the errors in the estimate of the gradient of the barrier functions for the EZO and SZO, respectively.  Let   
$\alpha_t = \min_{i = 1,\ldots,m}-f^i(x),~ \hat \alpha_t = \min_{i = 1,\ldots,m}-\hat f^i_{\de}(x), .$
Our approach is to first bound $\|\zeta_t\|_2$, and $\|\hat \zeta_t\|_2$ with high probability. Then we can construct a bound on the total number of steps and total number of measurements, required to find an $\eta$-approximate scaled KKT point. 

\textbf{EZO:}
Note that  $\zeta_t  
= \Delta^0_t + \eta\sum_{i=1}^m\frac{\Delta_t}{-f^i(x_t)}.$ As such, it is easy to see that if $\nu_t  = \min\{ \frac{\eta}{\sqrt{d}M},\frac{\alpha_t}{L},\frac{\alpha_t}{m\sqrt{d}M}\}$, then
\vspace{-0.5cm}
\begin{align}\label{b:1}
&\|\zeta_t\|_2 \leq \|\Delta^0_t\|_2 +  \frac{\eta}{\alpha_t}\sum_{i=1}^m\|\Delta^i_t\|_2 
\nonumber\\&\vspace{-0.2cm}
\leq \left(1+\frac{m\eta}{\alpha_t}\right)\frac{\sqrt{d}\nu_t M}{2} \leq \eta. 
\end{align}
\textbf{SZO:}
\begin{lemma}
	If $\nu_t  = \min\{ \frac{\eta}{\sqrt{d}M},\frac{\alpha_t}{L},\frac{\alpha_t}{m\sqrt{d}M}\}$, then
\vspace{-0.4cm}
\begin{align}\label{b:2}
\Prob \bigg\{\|\hat \zeta_t\|_2 \leq \eta\left(4+\frac{\eta}{L\sqrt{d}}\right) \bigg\} \geq 1-\delta.
	\end{align}
\end{lemma}
For the proof see Appendix E.\\
Let us next bound the number of iterations of the s0-LBM and 0-LBM.
\begin{lemma}\label{lemm:rate}
	After $T = \frac{C}{\eta^3}$ iterations of the 0-LBM or the s0-LBM algorithm with $C = $
	\vspace{-0.2cm}
	\begin{align*}
	 2(B_{\eta}(x_0) - \min_{x\in D}B_{\eta}(x))
	 \max \left\{mM\eta + 4L^2m, L\eta\right\}
	\end{align*}
 	 we obtain that for some $k \leq T$, $(x_k,\lambda_k)$ satisfy 
 	 \vspace{-0.2cm}
 	 \begin{align*}
 	 \hspace{-0.2cm}
 	 \|\nabla B_{\eta}(x_k)\|_2
 	  \hspace{-0.1cm} \leq \hspace{-0.1cm}
 	 \begin{cases} \hspace{-0.4cm}&\eta(1+\|\lambda_k\|_{\infty})+3\|\zeta_k\|_2\text{, 0-LBM} \\
 	 \hspace{-0.4cm}&\eta(1 + \|\hat \lambda_k\|_{\infty})+3\|\hat \zeta_k\|_2\text{, s0-LBM}
 	 \end{cases}.
 	 \end{align*} 
\end{lemma}
For the proof of \Cref{lemm:rate} see Appendix A. To guarantee that the point $(x_k, \lambda_k)$ above is an approximated scaled KKT point, $\|\zeta_k\|_2$ and $\|\hat \zeta_k\|_2$ need to be upper bounded by $O(\eta)$.  Such a bound can be obtained using (\ref{b:1}) and (\ref{b:2}). Consequently, we can also bound the total number of measurements for 0-LBM and s0-LBM for convergence to an approximate scaled KKT point.
\begin{theorem}\label{thm:0-LBM}
	After $T = \frac{C}{\eta^3}$ iterations of 0-LBM, 
 	 there exists an iteration $k  = \arg\min_{k\leq T} \gamma_t\| g_t\|^2_2$ such that $(x_k, \lambda_k) = \left(x_k, [\frac{\eta}{-f^i(x_k)}]\right)$  is an $\eta$-approximate scaled KKT point. The total number of measurements is 
	$N_T \leq T (d+1) =  O\left( \frac{d}{\eta^3} \right).$
\end{theorem} 
\begin{theorem}\label{thm:s0-LBM}
	After $T = \frac{C}{\eta^3}$ iterations of s0-LBM,
there exists an iteration $k  = \arg\min_{k\leq T} \gamma_t\|\hat g_t\|^2$ such that
 with probability greater than 
$1-\de$ 
	$(x_k, \hat \lambda_k)$ is a $4\eta$-approximate scaled KKT point. The total number of measurements is 
	$N_T \leq T (d+1) n_t = 
	O
	\left(\frac{d^3\sigma^2 \ln\frac{1}{\de}}{\eta^3(\min\{\eta,\min_t{\hat \alpha_t}\})^4}
	\right).$
\end{theorem}
Note that in s0-LBM, the total number of measurements is dependent on how close the iterations of the algorithm get to the boundary: $\min_t\{\hat \alpha_t\}$. For specific cases, 
we can prove that $\alpha_t$ are bounded from below by $\Omega(\eta)$, because the barrier gradient direction will be pointing out of the boundary. Moreover, for such cases $\|\lambda_t\|_{\infty} = \frac{\eta}{\min_t\{\hat \alpha_t\}}$ are bounded for all $t$, which means that we get an \textbf{unscaled} KKT point.
\begin{corollary}\label{thm:s0-LBM:unscaled}
Assume we have only one smooth constraint $ f^1(x)$, $m = 1$. Also, assume that $\exists F>0:\|\nabla f^1(x)\| \geq F$ for all $x\in \{D| f^1(x)\geq -\eta\}$ close enough  to the boundary.  Then, for $x_t$ generated by the s0-LBM we can guarantee $-f^1(x_t) \geq \frac{F}{L}\eta.$ Hence, after 
	$T = \frac{C}{\eta^3}$ iterations of the s0-LBM for $k  = \arg\min_{k\leq T} \gamma_t\|\hat g_t\|^2$  we find an \textbf{unscaled} $\eta$-approximate KKT point
$\Prob \left\{ \|\nabla B_{\eta}(x_k)\| \leq \eta
	\right\} \geq 1-\de.$
	The total number of measurements is 
	$N_T = O
	\left( \frac{d^{3}}{\eta^7}\right).$
\end{corollary}
For the proof see Appendix H. \\
It is also possible to extend this result for the case of multiple constraints, but then more regularity assumptions are needed.
\paragraph{Discussion.}
We use two notions for analysing our algorithms: the number of iterations $T$ and the number of measurements $N_T$. The number of steps of the first order method in \citet{hinder2019poly} was $T = O\left(\frac{1}{\eta^{3}}\right),$ similar to number of steps in our algorithm. In SafeOpt \citep{berkenkamp2016bayesian} the number of measurements 
$N_T$ is such that $\frac{N_T}{\bar\gamma_{N_T}} = O\left(\frac{1}{\eta^2}\right)$ (where $\bar\gamma_{N_T}$ is sub-linear on $N_T$), to achieve the accuracy $\eta$, however the complexity of each iteration is exponential in dimensionality $d$ due to solving a non-convex optimization problem in it. This makes the approach inapplicable for high dimensional problems. 
Safe Frank-Wolfe algorithm from \citet{usmanova2019safe} requires $N_T = O\left(\frac{d^2\ln\frac{1}{\de}}{\eta^2}\right)$  measurements 
for known convex objective and unknown linear constraints. 
 The convexity of the problem simplifies the situation for the number of iterations to $T = O(\frac{1}{\eta}).$ 
The fact that constraints are linear makes 
local information global to estimate the constraint set. In our case safety under non-convexity,  unknown constraints with noisy 1-point feedback, and convergence with high probability come at cost.

\section{EXPERIMENTS}
\begin{figure*}[ht!]
\begin{minipage}[l]{0.29\textwidth}
    \includegraphics[width =\textwidth]{
    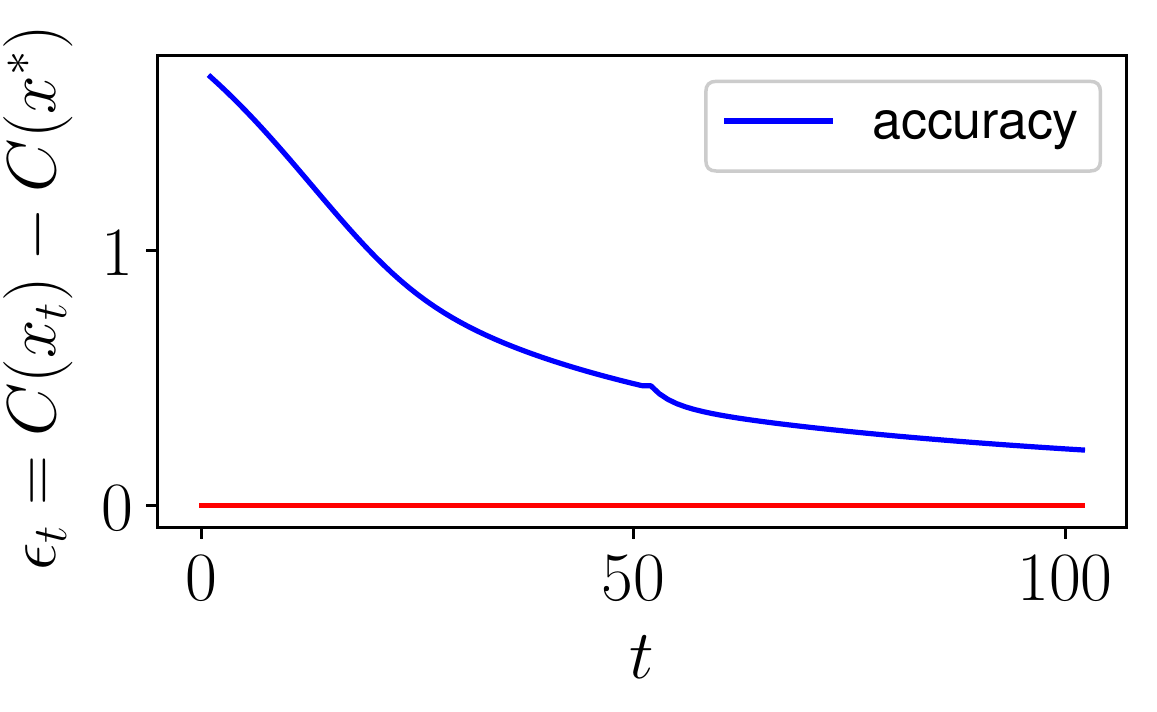}\\
    \includegraphics[width =\textwidth]{
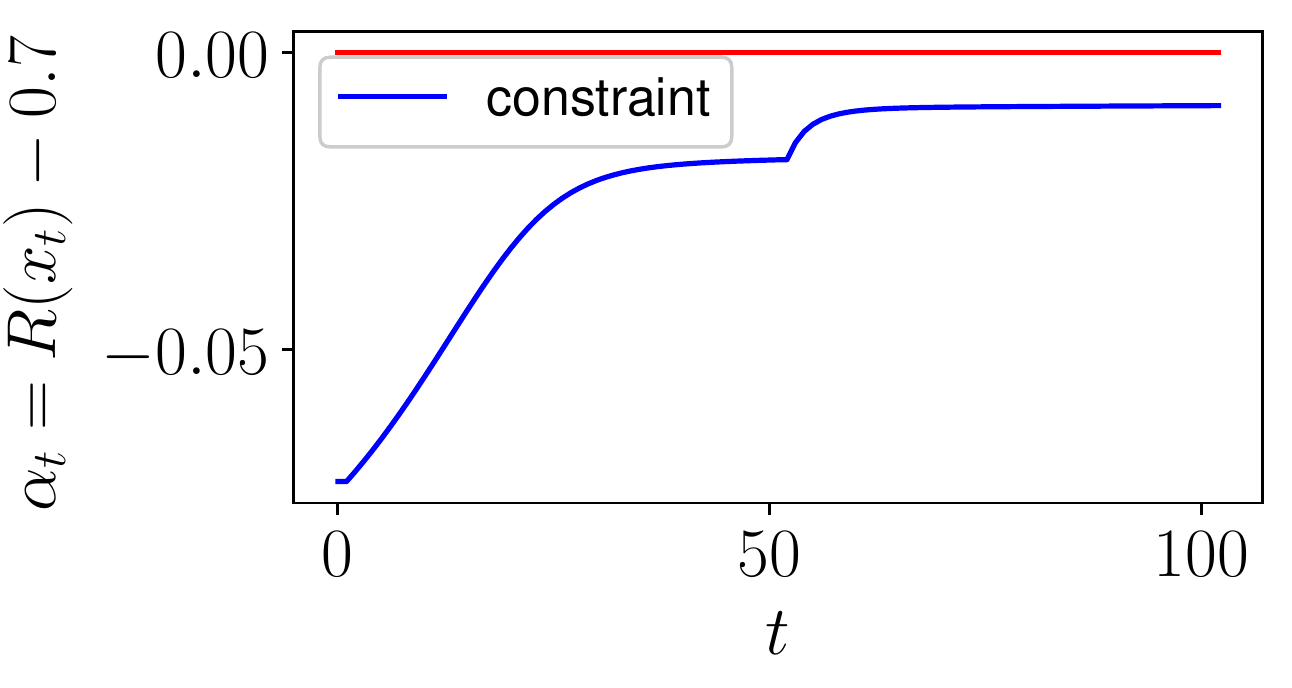} \\
    a)
\end{minipage}
\begin{minipage}[l]{0.19\textwidth}
    \includegraphics[width =\textwidth]{
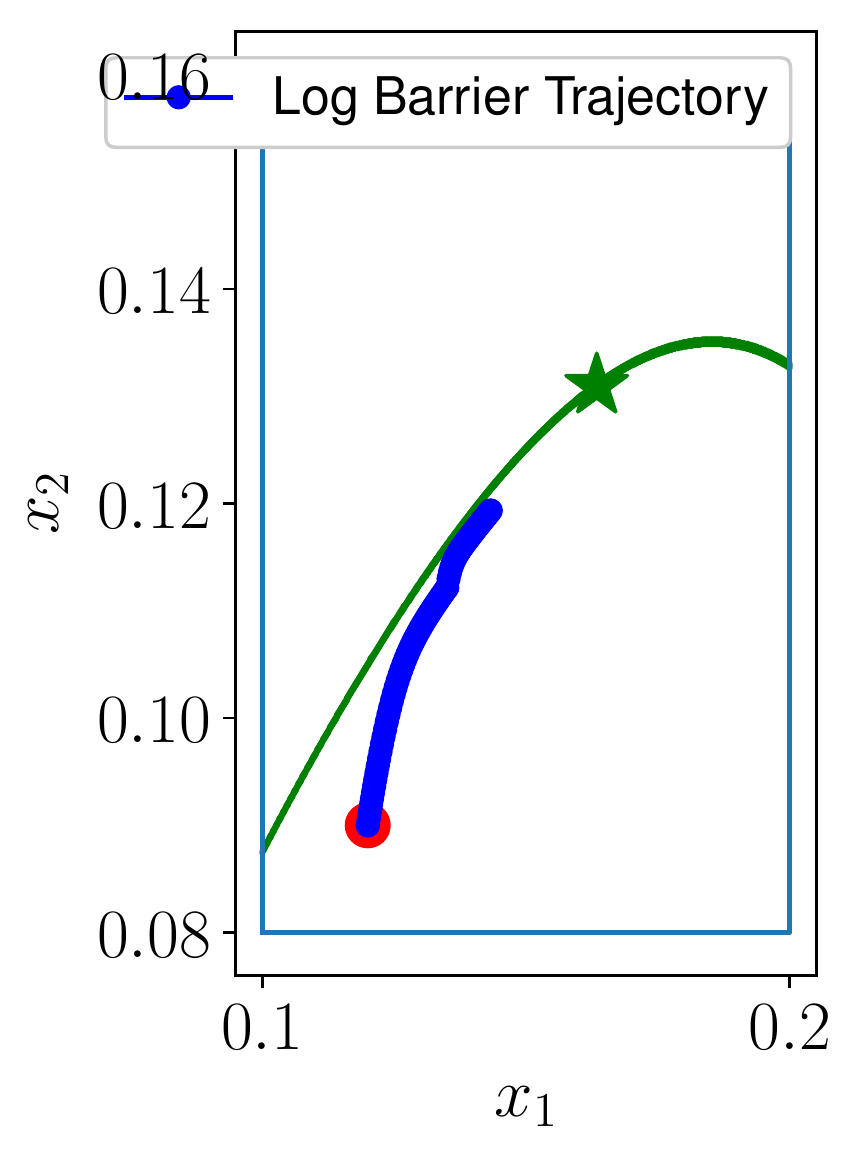}
\end{minipage}
\begin{minipage}[l]{0.29\textwidth}
\includegraphics[width =\textwidth]{
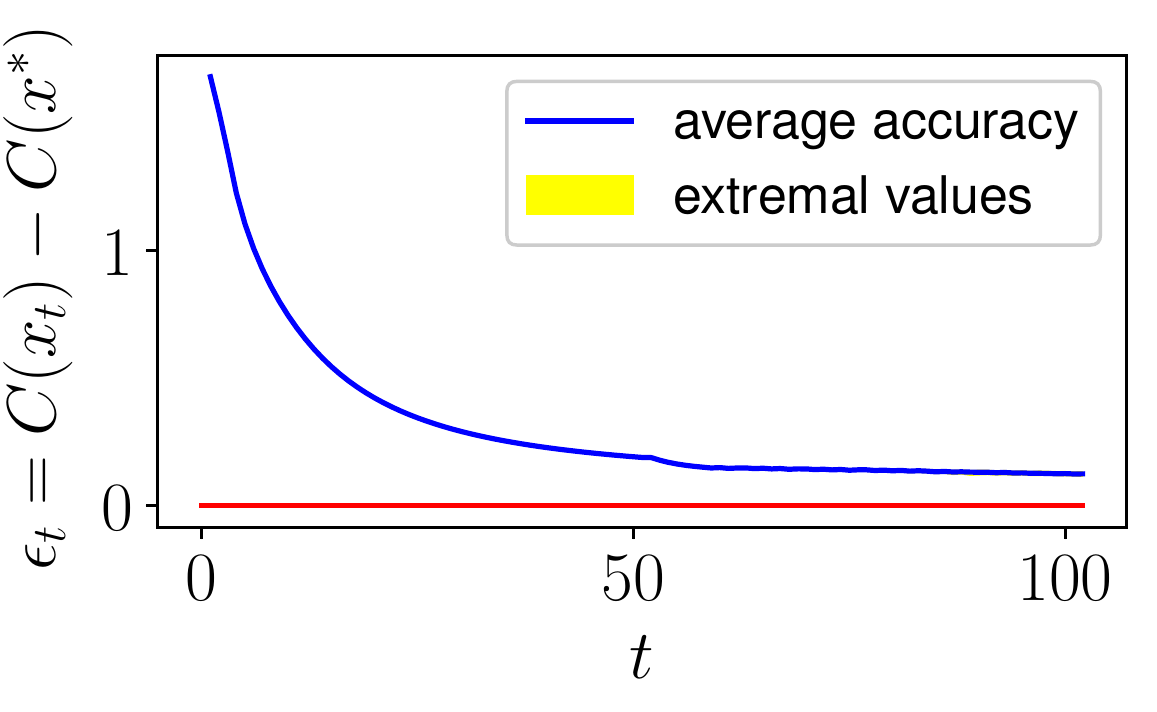}\\
    \includegraphics[width =\textwidth]{
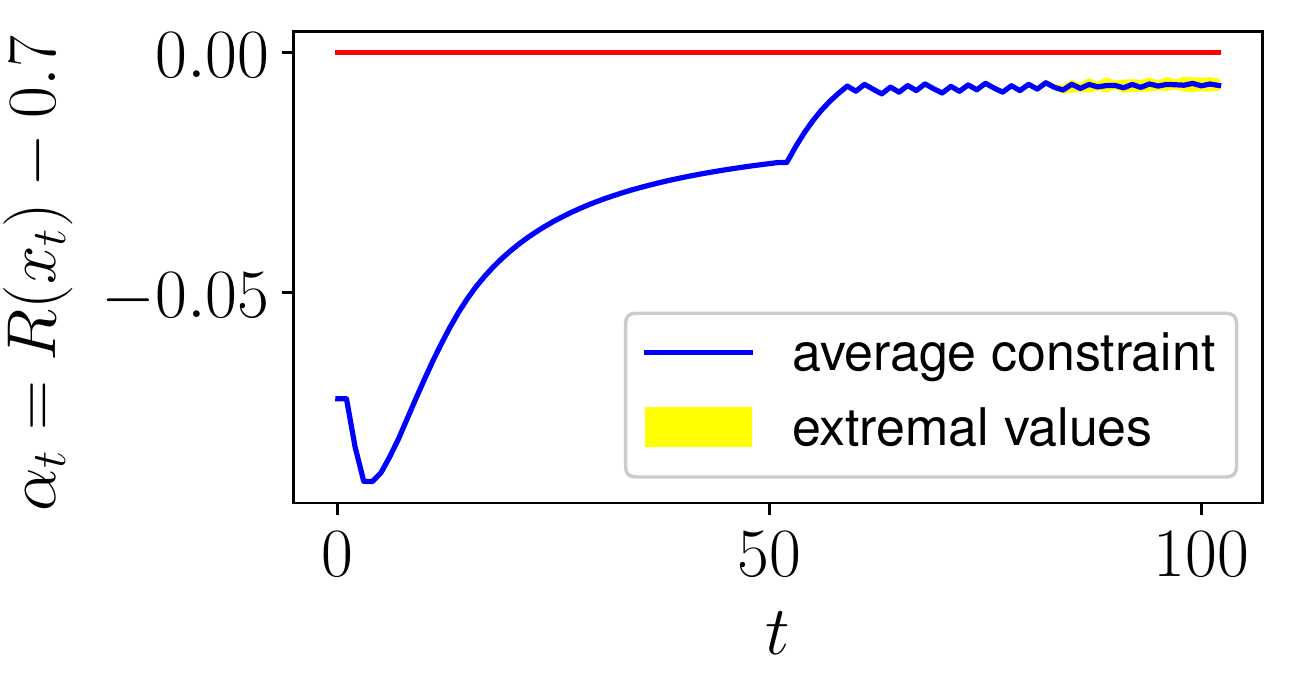} \\
    b)
\end{minipage}
\begin{minipage}[l]{0.19\textwidth}
    \includegraphics[width =\textwidth]{
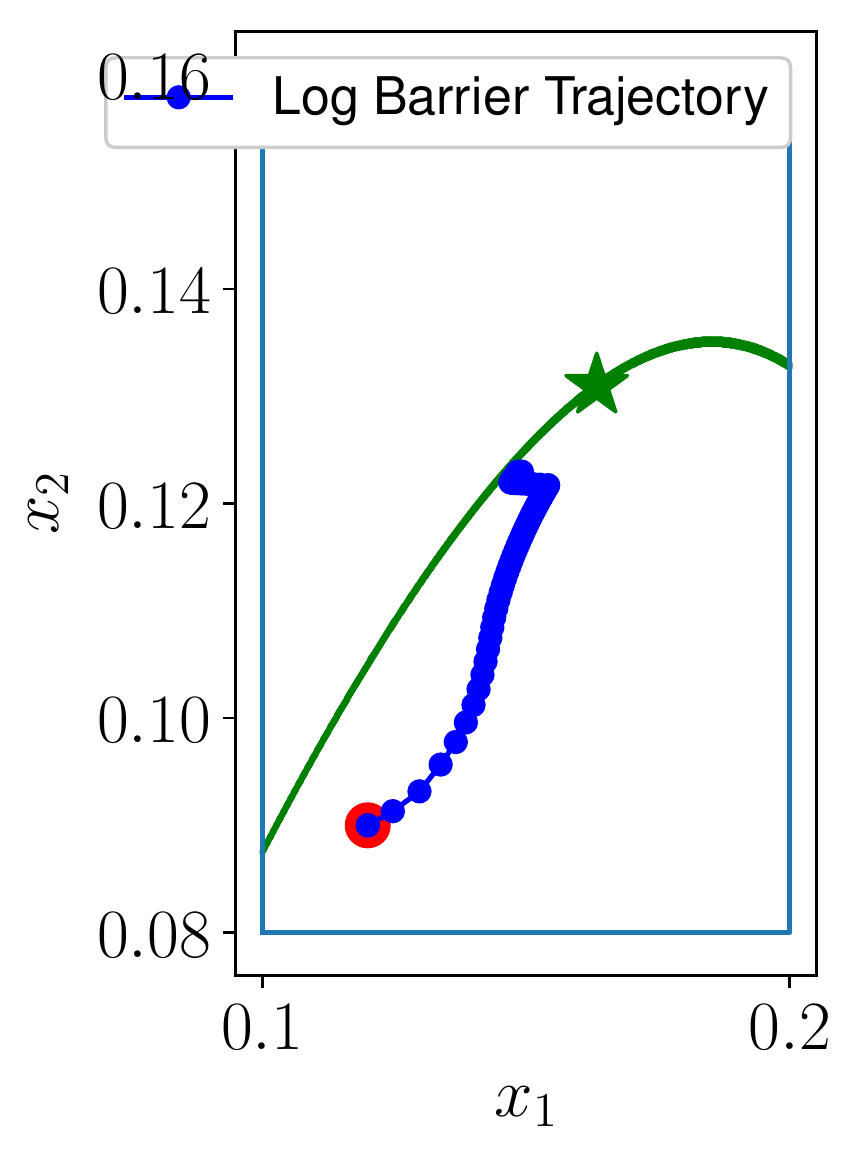} 
\end{minipage}
\caption{ a) Convergence of the objective, distance to the boundary, and optimization trajectories of 0-LBM. By the green star the optimal point is denoted. Blue points denote the optimization trajectory, red point denotes the starting point, green line is the boundary $R(x) = 0.7$.\\
b)  Convergence, distance to the boundary, and optimization trajectories of s0-LBM with SZO $\xi \sim \mathcal N(0,\sigma^2)$ added to the measurements $\hat R(x_t,\xi) =  R(x_t) + \xi$, $\hat C(x_t,\xi) =  C(x_t) + \xi$, 20 samples.
	}\label{F6}
\end{figure*}
We consider the scenario of a cutting machine \citep{maier2018turning} which has to produce certain tools and optimize the cost of production by tuning the turning process parameters such as the feed rate $f$ and the cutting speed $\nu_c$.  For the turning process we need to minimize a non-convex cost function $C(x)$, where the decision variable is $x = (\nu_c, f) \in \R^2$. The constraints include box constraints and a non-convex quality roughness constraint $R(x)$. 
 We perform realistic simulations, by using the  cost function and constraints estimated from hardware experiments 
with artificially added normally distributed noise $\xi \sim N(0,\sigma^2)$. 
   The obtained  non-convex smooth optimization problem with concave objective and convex constraints is:
  \vspace{-0.3cm}
\begin{align*}
&\min_{x\in \R^2} ~~~~~~~~~C(x) = t_c(x)\left(C_M + \frac{C_I}{T(x)} \right)\\
&\text{subject to~~ } 
R(x)\leq 0.7, x_1 = \nu_c \in [100, 200], \\&~~~~~~~~~~~~~~~~~~~~~~~~~~~~~~~~x_2 = f \in [0.08, 0.16].
\end{align*}
Here $t_c(x) = \frac{LD\pi}{\nu_c f},$ $T(x) = 127.5365 - 0.84629\nu_c - 144.21 f  + 0.001703 \nu_c^2 + 0.3656 \nu_c f,$  $R(x) = 0.7844 - 0.010035\nu_c + 7.0877 f + 0.000034 \nu_c^2 - 0.018969 \nu_c f,$ $C_I = 40$, $C_M = 50$. 
Note that we assume the box constraints to be known, i.e., not corrupted with noise. However, the roughness constraint $R(x)$ and the cost $C(x)$ are assumed to be unknown and we only can measure their values. Hence, this problem is an instance of the safe learning problem 
formulated in 
(\ref{problem}). More details are presented by \citet{maier2018turning}, who proposed to use Bayesian optimization to solve the problem. Although the Bayesian optimization used there indeed requires a small number of measurements, it is not safe and hence may require several measurements to be taken in the unsafe region. The roughness constraints are not fulfilled for unsafe measurements, i.e., the tools produced during unsafe experiments could not be realised in the market. That is why safety is necessary for this problem. 
Although there exist safe Bayesian optimization methods, they require strong prior knowledge in terms of suitable kernel function.
We solve barrier sub-problem (\ref{barrier:sub}) iteratively $K = 2$ times using s0-LBM with decreasing $\eta_{k+1} = \frac{\eta_k}{\mu}$, where we fix $\mu = 5$. We set $ \sigma = 0.01, L = 7, M = 5, \de = 0.01$ and re-scaled $\nu'_c = 0.001 \nu_c$ so that $\nu'_c \in [0.1,0.2].$ The starting point is $x_0 = (\nu'_c, f) = (0.15,0.09).$ In \Cref{F6}a) we show the performance of 0-LBM, and in b) we run 20 realizations s0-LBM. In all the realizations the method converges to a local optimum and 
the constraints are not violated.

\section{CONCLUSION}
We developed a zero-th order algorithm guaranteeing safety of the iterates and converging to a local stationary point. We provided its convergence analysis, which is comparable to existing zero-th order methods for non-convex optimization, and demonstrated its performance in a case study. 


\addcontentsline{toc}{section}{Bibliography}
\bibliography{bibliography}{}

\begin{thebibliography}{}

\bibitem[Bach and Perchet, 2016]{bach2016highly}
Bach, F. and Perchet, V. (2016).
\newblock Highly-smooth zero-th order online optimization.
\newblock In {\em Conference on Learning Theory}, pages 257--283.

\bibitem[Balasubramanian and Ghadimi, 2018]{balasubramanian2018zeroth}
Balasubramanian, K. and Ghadimi, S. (2018).
\newblock Zeroth-order (non)-convex stochastic optimization via conditional
  gradient and gradient updates.
\newblock In {\em Advances in Neural Information Processing Systems}, pages
  3455--3464.

\bibitem[Berkenkamp et~al., 2016]{berkenkamp2016bayesian}
Berkenkamp, F., Krause, A., and Schoellig, A.~P. (2016).
\newblock Bayesian optimization with safety constraints: safe and automatic
  parameter tuning in robotics.
\newblock {\em arXiv preprint arXiv:1602.04450}.

\bibitem[Birgin et~al., 2016]{birgin2016evaluation}
Birgin, E.~G., Gardenghi, J., Martinez, J.~M., Santos, S.~A., and Toint, P.~L.
  (2016).
\newblock Evaluation complexity for nonlinear constrained optimization using
  unscaled kkt conditions and high-order models.
\newblock {\em SIAM Journal on Optimization}, 26(2):951--967.

\bibitem[Carmon et~al., 2017]{carmon2017lower}
Carmon, Y., Duchi, J.~C., Hinder, O., and Sidford, A. (2017).
\newblock Lower bounds for finding stationary points i.
\newblock {\em Mathematical Programming}, pages 1--50.

\bibitem[Cartis et~al., 2014]{cartis2014complexity}
Cartis, C., Gould, N.~I., and Toint, P.~L. (2014).
\newblock On the complexity of finding first-order critical points in
  constrained nonlinear optimization.
\newblock {\em Mathematical Programming}, 144(1-2):93--106.

\bibitem[Facchinei et~al., 2017]{facchinei2017feasible}
Facchinei, F., Lampariello, L., and Scutari, G. (2017).
\newblock Feasible methods for nonconvex nonsmooth problems with applications
  in green communications.
\newblock {\em Mathematical Programming}, 164(1-2):55--90.

\bibitem[Ghadimi and Lan, 2013]{ghadimi2013stochastic}
Ghadimi, S. and Lan, G. (2013).
\newblock Stochastic first-and zeroth-order methods for nonconvex stochastic
  programming.
\newblock {\em SIAM Journal on Optimization}, 23(4):2341--2368.

\bibitem[Hinder and Ye, 2018]{hinder2018one}
Hinder, O. and Ye, Y. (2018).
\newblock A one-phase interior point method for nonconvex optimization.
\newblock {\em arXiv preprint arXiv:1801.03072}.

\bibitem[Hinder and Ye, 2019]{hinder2019poly}
Hinder, O. and Ye, Y. (2019).
\newblock A polynomial time log barrier method for problems with nonconvex
  constraints.
\newblock {\em arXiv preprint: https://arxiv.org/pdf/1807.00404.pdf}.

\bibitem[Maier et~al., 2018]{maier2018turning}
Maier, M., Rupenyan, A., Zwicker, R., Akbari, M., and Wegener, K. (2018).
\newblock Turning: Autonomous process set-up through bayesian optimization and
  gaussian process models 
  in turning.
\newblock {\em 12th CIRP Conference on Intelligent Computation in Manufacturing
  Engineering, Gulf of Naples, Italy}.

\bibitem[Schaal and Atkeson, 2010]{schaal2010learning}
Schaal, S. and Atkeson, C.~G. (2010).
\newblock Learning control in robotics.
\newblock {\em IEEE Robotics \& Automation Magazine}, 17(2):20--29.

\bibitem[Sui et~al., 2015]{sui2015safe}
Sui, Y., Gotovos, A., Burdick, J., and Krause, A. (2015).
\newblock Safe exploration for optimization with gaussian processes.
\newblock In {\em International Conference on Machine Learning}, pages
  997--1005.

\bibitem[Tang et~al., 2014]{tang2014feasible}
Tang, C.-m., Liu, S., Jian, J.-b., and Li, J.-l. (2014).
\newblock A feasible sqp-gs algorithm for nonconvex, nonsmooth constrained
  optimization.
\newblock {\em Numerical Algorithms}, 65(1):1--22.

\bibitem[Topkis and Veinott, 1967]{topkis1967convergence}
Topkis, D.~M. and Veinott, Jr, A.~F. (1967).
\newblock On the convergence of some feasible direction algorithms for
  nonlinear programming.
\newblock {\em SIAM Journal on Control}, 5(2):268--279.

\bibitem[Usmanova et~al., 2019]{usmanova2019safe}
Usmanova, I., Krause, A., and Kamgarpour, M. (2019).
\newblock Safe convex learning under uncertain constraints.
\newblock In {\em The 22nd International Conference on Artificial Intelligence
  and Statistics}, pages 2106--2114.

\bibitem[Yu and Ho, 2019]{yu2019zeroth}
Yu, Z. and Ho, D.~W. (2019).
\newblock Zeroth-order stochastic block coordinate type methods for nonconvex
  optimization.
\newblock {\em arXiv preprint arXiv:1906.05527}.

\end{thebibliography}
\bibliographystyle{apalike}
\clearpage 
\newpage
%
\twocolumn[
\section*{Appendix A}
\begin{proof}
First, recall that 
$\alpha_t = \min_{i = 1,\ldots,m}-f^i(x).$
Note that $\lambda_t^i = \frac{\eta}{-f^i(x_t)},$ 
$\|\lambda_t\|_{\infty} =\max_{i=1,\ldots,m}\frac{\eta}{-f^i(x_t)} =\frac{\eta}{\min_{i=1,\ldots,m} -f^i(x_t)}  = \frac{\eta}{\alpha_t}.$
Recall that the steps are given by
\begin{align}
    x_{t+1} =
    & x_t - \gamma_t g_t= x_t - \gamma_t \bigg( G^0(x_t, \nu_t)  + \eta \sum_{i=1}^t\frac{G^i(x_t, \nu_t)}{-f^i(x_t)}\bigg),
\end{align}
where the safe step size $\gamma_t $ is 
$\gamma_t = \min\left\{\frac{\min_i\{- \hat f_i(x_t)\}}{2L \|\hat g_t\|},\frac{1}{L_2(x_t)}\right\}.$ 
 In the paper \citep{hinder2019poly} the authors have shown that
\begin{align*}
L_2(x_t) &= M\left(1+\sum _{i=1}^m\frac{2\eta}{\{-f^i(x_t)\}}\right) + \sum_{i=1}^m\frac{4 \eta L}{ (-f^i(x_t))^2} = M(1+2\|\lambda_t\|_1) + \frac{4L\|\lambda_t\|_2^2}{\eta} 
\end{align*} 
represents a "local" Lipschitz constant of  $\nabla B_{\eta}(x)$ at the  point $x_t$. 
In particular 
in Lemma 1\citep{hinder2019poly} the authors have shown that
$|v^T \nabla^2 B_{\eta}(x_t+ \theta v)v| \leq L_2(x_t) $
for any $\theta \leq \frac{\alpha_t}{2L}$ and $v \in B(0,1)$. 
	Note that 
	\begin{align*}
	L_2(x_t) &= M\left(1+\sum _{i=1}^m\frac{2\eta}{\{-f^i(x_t)\}}\right) + \sum_{i=1}^m\frac{4 \eta L^2}{ (-f^i(x_t))^2} = M(1+2\|\lambda_t\|_1) + \frac{4L^2\|\lambda_t\|_2^2}{\eta} \\
	&\leq M\left(1+2\frac{m\eta}{\alpha_t}\right) + \frac{4L^2\eta m}{\alpha_t^2} = M\left(1+2m\|\lambda_t\|_{\infty}\right) + \frac{4L^2 m\|\lambda_t\|_{\infty}^2}{\eta}. \end{align*} 
	For the next inequalities we also use the fact that $\alpha_t = \frac{\eta}{\|\lambda_t\|_{\infty}}.$ Also we denote by $\hat \gamma_t = \gamma_t\|g_t\|.$
	Then, at each iteration of \Cref{safe-barrier-0}
	we have
	\begin{align*}
&B_{\eta}(x_t) - B_{\eta}(x_{t+1}) \geq - \gamma_t \la \nabla B_{\eta}(x_t),g_t\ra - \frac{1}{2} L_2(x_t) \gamma_t^2 \|g_t\|^2_2=
	\textit{[Taylor's theorem and local smoothness]}\\
&= \gamma_t \la \zeta_t,g_t\ra + \gamma_t\|g_t\|^2_2 - \frac{1}{2} L_2(x_t) \gamma_t^2 \|g_t\|^2_2
	\geq \textit{by definition } \gamma_t = \min\left\{\frac{1}{\hat L_2(x_t)},\frac{\hat\alpha_t}{2L\|g_t\|}\right\}\\
	&\geq - \hat \gamma_t \| \zeta_t\|_2 + \frac{{\gamma_t}}{2}\|g_t\|^2_2  \geq  - \hat \gamma_t \| \zeta_t\|_2 + \frac{\hat \gamma_t }{2}\|g_t\|_2 \geq \frac{\hat \gamma_t }{2}(\|\nabla B_{\eta}(x_t)\|_2 - 3 \| \zeta_t\|_2).
\end{align*}
 Hence, we can derive the following bound for all $t$ for which 
$\|g_t\|_2 \geq \|g_t\|_2 - 2\|\zeta_t\|\geq  \eta(1+\frac{\eta}{\hat \alpha_t})$, 
we have
	\begin{align*}
	& B_{\eta}(x_0) - \min_{x\in D} B_{\eta}(x) \geq B_{\eta}(x_0) - B_{\eta}(x_T) = \sum_{t=0}^{T} \left( B_{\eta}(x_t) - B_{\eta}(x_{t+1})\right) \geq  \sum_{t=1}^{T}\frac{\hat \gamma_t }{2}(\|\nabla B_{\eta}(x_t)\|_2 - 3 \|\zeta_t\|_2)\geq
	\\
	&\geq \frac{1}{2}\sum_{t=1}^T\min\left\{\frac{\|g_t\|_2}{\hat L_2(x_t)},\frac{\hat\alpha_t}{2L}\right\}\eta\left(1+\frac{\eta}{\alpha_t}\right) \geq \frac{1}{2}\sum_{t=1}^T\min\left\{\frac{\eta^2\left(1+\frac{\eta}{\alpha_t}\right)^2}{\hat L_2(x_t)},\frac{\hat\alpha_t \eta\left(1+\frac{\eta}{\alpha_t}\right)}{2L}\right\}\\
	&\geq \frac{1}{2}\sum_{t=1}^T\min\left\{\frac{\eta^2\left(1+\frac{\eta}{\alpha_t}\right)^2}{mM\left(1+\frac{\eta}{\alpha_t}\right) + \frac{4L^2m}{\eta}(\frac{\eta^2}{\alpha_t^2})},\frac{ \eta^2}{L}\right\}\geq \frac{T}{2}\min\left\{\frac{\eta^3}{mM\eta + 4L^2m},\frac{\eta^2}{L}\right\}.
\end{align*}
	\end{proof}
]

	\twocolumn[
	\begin{proof}\textit{(Continuation)}
	Hence after 
\begin{align*}
T \leq 2(B_{\eta}(x_0) - \min_{x\in D} B_{\eta}(x))\max \left\{\frac{mM\eta + 4L^2m}{\eta^3},\frac{L}{\eta^2}\right\}.
\end{align*} iterations we will find $\|g_k\|_2-2\|\zeta_k\|_2\leq \eta\left(1+\frac{\eta}{\hat \alpha_k}\right).$
\end{proof}
\section*{Appendix G}
\begin{proof}
For 0-LBM 
$$\|\nabla B_{\eta}(x_k)\|_2\leq \|g_k\|_2 + \|\zeta_k\|_2 \leq  \eta\left(1+\frac{\eta}{ \alpha_k}\right)+3\|\zeta_k\|_2 .$$
For s0-LBM
$$\|\nabla B_{\eta}(x_k)\|_2\leq \|\hat g_k\|_2 + \|\hat \zeta_k\|_2 \leq  \eta\left(1+\frac{\eta}{ \hat \alpha_k}\right)+3\|\hat \zeta_k\|_2 \leq \eta\left(1+\frac{\eta}{ \alpha_k}\right)+3\|\hat \zeta_k\|_2.$$
If $\|\zeta_k\|_2 \leq \eta,$ then $\|\nabla B_{\eta}(x_k)\|_2\leq \eta(4+\frac{\eta}{\hat\alpha_k}).$

\begin{align*}
&\hat\lambda_i(-f^i(x_k)) = \frac{\eta}{-\hat f^i_{\de}(x_k)} (-f^i(x_k)) = \eta + \frac{\eta}{-\hat f^i_{\de}(x_k)} (\hat f^i(x_k)  -f^i(x_k))\leq \eta + \frac{\eta \e_k}{-\hat f^i_{\de}(x_k)}\\
& \leq \eta + \frac{3\eta\nu_k^2M^2}{-\hat f^i_{\de}(x_k)}\leq \eta + 3\eta\nu_k M \leq 4\eta.
\end{align*}
Then
\begin{align*}
&\|\nabla B_{\eta}(x_k)\| \leq \eta(4+\|\hat\lambda_k\|_{\infty}),\\
& \hat\lambda_k^i (-f^i(x_t)) \leq 4\eta,\\
& -f^i(x_t), \hat\lambda_k^i \geq 0,
	\end{align*} i.e.,  
	$(x_k, \hat \lambda_k)$ is a $4\eta$-approximate scaled KKT point.
	\end{proof}
]

\twocolumn[


\section*{Appendix E}
\begin{proof}
Let us denote by $$g'_t = \hat G^0 (x_t, \nu_t, \xi) + \eta \sum_{i = 1}^m \frac{\hat G^i (x_t, \nu_t, \xi)}{-f^i(x_t)}$$ and $\zeta'_t =  g'_t -  \nabla B_{\eta}(x_t).$ 
Note that by \Cref{f_up2_SZO} we have $0 \leq \hat f^i(x) - f^i(x) \leq \e_t$  with high probability by construction. 
	Then, based on (\ref{g_t_SZO}), with probability $\geq 1-\de$ \begin{align*}
	\|\hat \zeta_t\|_2 &= \|\hat g_t - \nabla B_{\eta}(x_t)\|_2 \leq \|\hat g_t - g'_t\|_2 + \|g'_t - \nabla B_{\eta}(x_t)\|_2  \\
	&= \eta \left\| \sum_{i=1}^m \hat G^i(x_t, \nu_t, \xi)\left(\frac{1}{-\hat f^i(x_t)} - \frac{1}{-f^i(x_t,)} \right)\right\|_2 + \|\zeta'_t\|_2 \\
	&
	\leq \eta \left\| \sum_{i=1}^m\frac{ \hat G^i(x_t, \nu_t, \xi)}{-f^i(x_t)} \frac{ \e_t}{-\hat f^i(x_t)} \right\|_2 + \|\zeta'_t\|_2\leq \frac{\eta}{\alpha_t \hat \alpha_t} \left\| \sum_{i=1}^m ( \nabla f^i(x_t) +\hat \Delta^i_t) \e_t \right\|_2 + \|\zeta'_t\|_2\\
	&\leq
	\frac{\eta}{\hat \alpha_t^2} \left ( mL +\sum_{i = 1}^m \|\hat \Delta^i_t\|_2 \right) \e_t +  \sqrt{\|\hat \Delta^0_t\|^2_2+\sum_{i = 1}^m \frac{\eta^2}{ \alpha_t^2}\|\hat \Delta^i_t\|^2_2} 
	\end{align*}
	
	Recall that by \Cref{lemma:1} with probability $\geq 1-\de$, if $n_t = \frac{8\sigma^2 \ln \frac{1}{\de}}{3M^2 \nu_t^4}.$ 

	$$\|\hat \Delta_t\|^2_2 \leq  \frac{d\nu_t^2 M^2}{4} + 2d\sigma^2 \frac{\ln 1/\delta}{n_t\nu_t^2}
 \leq d\nu_t^2 M^2$$
 $$\e_t \leq 2 \frac{\sigma \sqrt{\ln \frac{1}{\de}}}{\sqrt{n_t}} \leq M\nu_t^2$$
Then 
	\begin{align*}
	\|\hat \zeta_t\|_2 &\leq 
	\frac{\eta}{\hat \alpha_t^2}M\nu_t^2  \left ( mL +m\sqrt{d}\nu_t M\right) +  \sqrt{d\nu_t^2 M^2+m \frac{\eta^2}{ \alpha_t^2}d\nu_t^2 M^2}
	.\end{align*}
		If $\nu_t  = \min\{ \frac{\eta}{\sqrt{d}M},\frac{\alpha_t}{L},\frac{\alpha_t}{m\sqrt{d}M}\}$, then

\begin{align}\label{b:2}
\Prob \bigg\{\|\hat \zeta_t\|_2 \leq \eta\left(4+\frac{\eta}{L\sqrt{d}}\right) \bigg\} \geq 1-\delta.
	\end{align}
\end{proof}
]

\twocolumn[
\section*{Appendix F}
\begin{proof}
    Denote by $\xi^i_{t,j}:= \frac{\sum_{l=1}^{n_t}\xi^i_{jl}}{n_t}$ at iteration $t$.	The deviation is bounded as follows.
	\begin{align*}
	\|\hat \Delta_t^i\|_2  = \|\hat G^i(x_t, \nu_t, \xi^i_{(n_t)}) - \nabla f^i(x_t)\| &= \sqrt{\sum_{j = 1}^d\left[\frac{f^i(x_t + \nu_t e_j)+\xi^i_{t,j} - f^i(x_t)-\xi^i_{t,0}}{\nu_t}   - \la \nabla f^i(x_t), e_j\ra\right]^2\|e_j\|^2} \leq\\
	&
	\leq \sqrt{\frac{d\nu^2 M^2}{4} + \frac{1}{\nu_t^2}\sum_{j=1}^d[ \xi^i_{t,j} - \xi^i_{t,0}]^2} \leq \sqrt{\frac{d\nu_t^2 M^2}{4} + \frac{1}{\nu_t^2}\sum_{j=1}^d[\xi^i_{t,j} - \xi^i_{t,0}]^2} 
	\end{align*}
	Hence, the deviation is upper bounded with high probability  by 
	\begin{align*}
	&\Prob \left\{\|\hat \Delta_t^i\|_2   \leq  \sqrt{\frac{d\nu_t^2 M^2}{4}  + 2d\sigma^2 \frac{\ln 1/\delta}{n_t \nu_t^2}}\right\} \geq 1-\de
	\end{align*}
	
	Hence, $n_t$ has to be chosen by trading off this 2 terms. Setting 
	$n_t \geq \frac{8\sigma^2 \ln \frac{1}{\delta}}{3\nu_t^4 M^2}$, we obtain:
	
	\begin{align}\label{bias2}
	&\Prob \left\{\|\hat \Delta_t^i\|_2  \leq   \sqrt{d}\nu_t M  \right\} \geq 1-\de
	\end{align}
\end{proof}
\section*{Appendix D}
\begin{proof}
Recall $\xi^i_{t,j}:= \frac{\sum_{l=1}^{n_t}\xi^i_{jl}}{n_t}.$ Note that
\begin{align*}
    f^i(x_t) = F^i(x_t, \xi^i_{t,0}) - \xi^i_{t,0} = \frac{1}{n_t}\sum_{l = 1}^{n_t} F^i(x_t, \xi^i_{0l}) - \frac{1}{n_t}\sum_{l = 1}^{n_t}\xi^i_{0l} 
\end{align*}
From the sub-Gaussian property of $\xi^i_0$ and Hoeffding inequality, we have that:
$$\Prob \left\{\sum_{l = 1}^{n_t} \frac{\xi^i_{0l} - \E \xi^i_{0l}}{n_t} \geq b\right\} = 
\Prob \left\{\frac{1}{n_t}\sum_{l = 1}^{n_t} \xi^i_{0l} \geq b\right\}\leq \exp{\frac{-n_t b^2}{\sigma^2}} = \de.$$
Then $b = \frac{\sigma\sqrt{\ln \frac{1}{\de}}}{\sqrt{n_t}},$ $\Prob \left\{\frac{1}{n_t}\sum_{l = 1}^{n_t} \xi^i_{0l} \leq \frac{\sigma\sqrt{\ln \frac{1}{\de}}}{\sqrt{n_t}}\right\}\geq 1-\de.$
Similarly, $$\Prob \left\{\frac{1}{n_t}\sum_{l = 1}^{n_t} -\xi^i_{0l} \leq \frac{\sigma\sqrt{\ln \frac{1}{\de}}}{\sqrt{n_t}}\right\}\geq 1-\de.$$
Hence, if $\hat f^i_{\de} = \frac{1}{n_t}\sum_{l = 1}^{n_t} F^i(x_t, \xi^i_{0l}) + \frac{\sigma\sqrt{\ln \frac{1}{\de}}}{\sqrt{n_t}},$ then $\Prob \left\{\hat f^i_{\de} - 2\frac{\sigma\sqrt{\ln \frac{1}{\de}}}{\sqrt{n_t}} \leq f^i_{\de} \leq \hat f^i_{\de}\right\} \geq 1-\de.$
\end{proof}
 ]
	\twocolumn[
\section*{Appendix H.}
\begin{proof}
Assume we have only one smooth constraint $ f^1(x)$, $m = 1$. Also, assume that close enough  to the boundary  $\forall x\in D: f^1(x)\geq -\eta$ we have $\|\nabla f^1(x)\| \geq F.$ \\
Assume that for some $x_t$ $-f^1(x_t) \leq \frac{F^2}{L^2 + \eta}\eta.$
At that point 
$$F\leq\|\nabla f^1(x_t)\|\leq L$$
and 
$$\|\nabla f^1(x_t)\|\leq L$$
Then the step direction $-\hat g_t = -\nabla B_{\eta} + \hat \zeta_t$ is such that 
\begin{align*}
&\la\hat g_t, \nabla f^1(x_t) \ra = \la\nabla B_{\eta} + \hat \zeta_t, \nabla f^1(x_t) \ra  =\\
&\la\nabla f^0(x_t) +\eta\frac{\nabla f^1(x_t)}{-f^1(x_t)} + \hat \zeta_t, \nabla f^1(x_t) \ra \geq \la\nabla f^0(x_t), \nabla f^1(x_t)\ra + \frac{\|\nabla f^1(x_t)\|^2}{F^2/(L^2+\eta L)} - \eta L \geq L^2 + \eta L - L^2 - \eta L \geq 0.
\end{align*}
Then if $\gamma_t \leq \frac{1}{M}$ and $-f^1(x_t) \leq \frac{F^2}{(L^2 + \eta L)}\eta$, the next step $f^1(x_{t+1}) < f^1(x_{t})$ will decrease. On the other hand, if $-f^1(x_t) \geq \frac{F^2}{(L^2 + \eta L)}\eta$ for one step it cannot decrease more than twice $f^1(x_{t+1}) \geq  \frac{F^2}{2(L^2 + \eta L)}\eta.$ Hence, for $x_t$ generated by s0-LBM we can guarantee $-f^1(x_t) \geq \frac{F^2}{2(L^2 + \eta L)}\eta.$
\end{proof}
\section*{Appendix I.}
\begin{proof}
Using smoothness $$f(x) \geq f(y) + \la \nabla f(y), x-y\ra - \frac{M}{2}\|x-y\|_2^2$$ we can obtain 
\begin{align*}
\|\Delta_t^i\|_2&=
\sqrt{\sum_{j = 1}^d\left[\frac{f^i(x_t + \nu_t e_j) - f^i(x_t)}{\nu_t}  - \la \nabla f^i(x_t), e_j\ra\right]^2\|e_j\|_2^2} \\
&\leq \sqrt{\sum_{j = 1}^d\left[\frac{f^i(x_t + \nu_t e_j) - f^i(x_t) - \la \nabla f^i(x_t), \nu_t e_j\ra}{\nu_t} \right]^2\|e_j\|_2^2}\\
&\leq  \sqrt{\sum_{j = 1}^d\left[\frac{M\nu_t^2}{2\nu_t}\right]^2}
\leq
\frac{\sqrt{d}\nu_t M}{2}.
\end{align*}
\end{proof}
]

\end{document}